\documentclass[11pt,a4paper]{article}
\usepackage{amsmath,amsthm,amssymb,amscd}
\usepackage{latexsym}
\usepackage{indentfirst}
\usepackage{bibentry}
\usepackage{textcomp}
\usepackage{float}
\usepackage{cases}
\usepackage{multirow}
\usepackage{booktabs}
\usepackage{graphicx}
\usepackage{color}
\usepackage[numbers,sort&compress]{natbib}
\usepackage{booktabs}
\usepackage{authblk}
\usepackage[numbers]{natbib}
\usepackage{subfigure}
\usepackage{caption}
\usepackage{color}
\usepackage{tikz}
\usetikzlibrary{shapes.geometric, arrows}
\setcitestyle{open={},close={}}
\advance \topmargin by -\headheight
\advance \topmargin by -\headsep
\evensidemargin \oddsidemargin
\makeatletter \@addtoreset{equation}{section}

\makeatother
\newtheorem{thm}{Theorem}[section]

\newtheorem{lem}{Lemma}[section]

\theoremstyle{definition}
\newtheorem{rem}{Remark}[section]

\tikzstyle{startstop} = [rectangle,rounded corners, minimum width=3cm,minimum height=1cm,text centered, draw=black,fill=red!30]
\tikzstyle{io} = [trapezium, trapezium left angle = 70,trapezium right angle=110,minimum width=3cm,minimum height=1cm,text centered,draw=black,fill=blue!30]
\tikzstyle{process} = [rectangle,minimum width=3cm,minimum height=1cm,text centered,text width =3cm,draw=black,fill=orange!30]
\tikzstyle{decision} = [diamond,minimum width=3cm,minimum height=1cm,text centered,draw=black,fill=green!30]
\tikzstyle{arrow} = [thick,->,>=stealth]

\begin{document}

\title{Orthogonal polynomials: from Heun equations to Painlev\'{e} equations}
\author{Mengkun Zhu$^{1}$\footnote{zmk@qlu.edu.cn}, Yuting Chen$^{1}$\footnote{cytldy@163.com}, Jianduo Yu$^{2}$\footnote{Corresponding author: yujd0203@163.com}, Chuanzhong Li$^{3}$\footnote{lichuanzhong@sdust.edu.cn}\\
\small $^{1}$ School of Mathematics and Statistics, Qilu University of Technology (Shandong Academy of Sciences),\\
\small Jinan, 250353, China\\
\small $^{2}$ School of Mathematics and Statistics, Ningbo University, Ningbo, 315211, China\\
\small $^{3}$ College of Mathematics and Systems Science, Shandong University of Science and Technology,\\
\small Qingdao, 266590, China}
\renewcommand\Authands{ and }

\date{}
\maketitle
\begin{abstract}
In this paper, we {\color{black}study four kinds of polynomials orthogonal with the singularly perturbed Gaussian weight $w_{\rm SPG}(x)$, the deformed Freud weight $w_{\rm DF}(x)$, the jumpy Gaussian weight $w_{\rm JG}(x)$, and the Jacobi-type weight $w_{\rm {\color{black}JC}}(x)$. The second order linear differential equations satisfied by these orthogonal polynomials and the associated Heun equations are presented. Utilizing the method of isomonodromic deformations from [J. Derezi\'{n}ski, A. Ishkhanyan, A. Latosi\'{n}ski, SIGMA 17 (2021), 056], we transform these Heun equations into Painlev\'{e} equations. It is interesting that the Painlev\'{e} equations obtained by the way in this work are same as the results satisfied by the related three term recurrence coefficients or the auxiliaries studied by other authors. In addition, we discuss the asymptotic behaviors of the Hankel determinant generated by the first weight, $w_{\rm SPG}(x)$, under a suitable double scalings for large $s$ and small $s$, where the Dyson's constant is recovered.}
\end{abstract}



\section{Introduction}
\noindent Painlev\'{e} equations (Painlev\'{e} I-VI) are a class of nonlinear second-order ordinary differential equations(ODE), with integrability properties, named after French mathematician Paul Painlev\'{e}, which have many important mathematical and physical applications, such as critical phenomena in statistical physics, conformal field theory and topological field theory in quantum field theory, and black hole physics in astrophysics. Furthermore, as an important mathematical object, the study of Painlev\'{e} equations also has its own mathematical value. For example, in algebraic geometry, Painlev\'{e} equations are closely related to algebraic power series and automorphic representations. The specific form and properties of Painlev\'{e} equations can be found in the literatures [\cite{Its,Iwas}], as well as many famous mathematician's related works, such as Its, Krichever, Clarkson, Kapaev, Joshi, Jimbo, Miwa and Okamoto, etc.

Orthogonal polynomials(OPs) also have been widely researched and applied in mathematical physics, random matrices, Riemann-Hilbert problem, approximation theory, Painlev\'{e} equations, and so on, which can be found in a lot of famous research work by Szeg\"{o}, Mehta, Magnus, Deift, Ismail, Clarkson, etc, {\color{black}see [\cite{ClarksonJAT},\cite{a19},\cite{Mehta},\cite{Szego}]} and the references therein. Let $P_n(x)$ be the monic orthogonal polynomials, of degree $n$, associated with a weight function $w(x)$ supported on $\mathrm{L}\subset\mathbb{R}$. Chen and Ismail derived the raising and lowering differential recurrence relations (were usually called ladder operators) for these orthogonal polynomials in [\cite{chen1997}]. Under the ladder operators, they found a second order differential equation satisfied by these polynomials, which reads
\begin{equation*}
\frac{d^2P_n(x)}{dx^2}-\left(v'(x)+\frac{A_n'(x)}{A_n(x)}\right)\frac{dP_n(x)}{dx}+\left(B_n'(x)-B_n(x)\frac{A_n'(x)}{A_n(x)}
+\sum_{j=0}^{n-1}A_j(x)\right)P_n(x)=0,
\end{equation*}
where $v(x):=-\ln w(x)$ and
\begin{equation*}
A_n(x):=\frac{1}{h_n}\int_{\mathrm{L}}\frac{v'(x)-v'(y)}{x-y}P_n^2(y)w(y)dy,
\end{equation*}
\begin{equation*}
B_n(x):=\frac{1}{h_{n-1}}\int_{\mathrm{L}}\frac{v'(x)-v'(y)}{x-y}P_n(y)P_{n-1}(y)w(y)dy,
\end{equation*}
with $h_n$ indicates the square of $L^2-$norm of the monic polynomial $P_n(x)$. $P_{n}(x)$ and $w(x)$ depend on $\alpha,t$ or $a$, but to simplify notation we do not always display them, unless it is just required. The ladder operator method has been widely applied to study the OPs and the spectral analysis in unitary random matrix ensembles for various weight functions, see e.g. [\cite{Basor},\cite{Min},\cite{MC},\cite{Wang}]. They established the relation of the three-term recurrence coefficients(or the relevant quantities) and the Hankel determinants, generated by Gaussian, Laguerre, and Jacobi weights, as well as some of their deformations, with the Painlev\'{e} (including the $\sigma-$ form) equations.

The general Heun equation with four regular singular points ${0, 1, a, \infty}$ in the complex plane is the most general second-order linear Fuchsian ordinary differential equation. It is given by
\begin{equation}\label{1.0}
\frac{d^2y}{dx^2}+\bigg(\frac{\gamma}{x}+\frac{\delta}{x-1}+\frac{\varepsilon}{x-a}\bigg)\frac{dy}{dx}+\frac{\alpha\beta x-q}{x(x-1)(x-a)}y=0,
\end{equation}
with the condition $\gamma+\delta+\varepsilon=\alpha+\beta+1$ and it is a standard form of second order linear differential equation. It can reduce to the four confluent types[\cite{Heun}]: {\color{black}the confluent Heun, double-confluent Heun, biconfluent Heun, and tri-confluent Heun equations.} The general Heun function with the first order derivative satisfies a second order Fuchsian differential equation has five regular singular points. The linear equation for the function {\color{black}$v(x) =x^{\gamma}(x-1)^{\delta}(x-a)^{\varepsilon} dy/dx$}, where $y =y(x)$ is a solution of \eqref{1.0}, is given by direct computations[\cite{new2018}]
\begin{equation}\label{1.01}
\frac{d^2v}{dx^2}+\bigg(\frac{1-\gamma}{x}+\frac{1-\delta}{x-1}+\frac{1-\varepsilon}{x-a}{\color{black}+}\footnote{\color{black}In the original formula of [\cite{new2018}, Eq.(3)], the minus sign here should be changed to the plus sign.}\frac{\alpha\beta}{\alpha\beta x-q}\bigg)\frac{dv}{dx}+\frac{\alpha\beta x-q}{x(x-1)(x-a)}v=0,
\end{equation}
It is obvious that an additional singularity at $x = q/(\alpha\beta)$ involving the accessory parameter is added. \eqref{1.01} is called a deformed Heun equation [\cite{Slavjanov}]. {\color{black}Filipuk et al.} [\cite{Filipuk}] compared the linear second order differential equations with one more singularities for the Heun derivatives with linear differential equations isomonodromy deformations of which are described by the Painlev\'{e} equations. Heun differential equation and its confluent types appear in several systems of physics
including quantum mechanics, general relativity, crystal transition and fluid dynamics{\color{black}[\cite{Craster},\cite{Olver}-\cite{Slavjanov}]}. It is shown in [\cite{Zhan2}] that the second order linear differential equations satisfied by $P_{n}(x)$ are particular cases of Heun equations when $n$ is large. In some sense, monic orthogonal polynomials generated by deformed weights are solutions of a variety of Heun equations. The degree-$n$ polynomials $P_{n}(x)$ orthogonal with respect to weights is asymptotically equivalent to the biconfluent Heun equations as $n\rightarrow\infty$ derived by {\color{black}Wang et al. in} [\cite{Wang}].

The relationship between Heun class equations and the Painlev\'{e} equations is a hot topic and widely studied in recent years.
In [\cite{Fuchs1},\cite{Fuchs2}], Fuchs discovered a remarkable relation between the Heun equations and the Painlev\'{e} VI. He added
an extra apparent singularity at $x =g/(\alpha\beta)$ to \eqref{1.01}. The apparent singularity is presented
in the equation but is absent in the solution. The position of the apparent singularity is deformed with the regular singularity $t$ in \eqref{1.0}. Inspired by the works of Fuchs, there are many researches studied the deformation of the Heun class of equations by adding an apparent singularity. It was shown that each Painlev\'{e} equation can be considered as the isomonodromy deformation condition for a deformed equation
of the Heun class. Slavyanov in [\cite{Yu}] also discovered that the Heun class equations can be regarded as the quantization of the classical Hamiltonian of the Painlev\'{e} equations. It was observed in [\cite{Xia}] that the Heun class equations can be obtained as limits of the linear systems associated with the Painlev\'{e} equations when the Painlev\'{e} transcendents go to one of the actual singular points of the linear systems. {\color{black}Derezi\'{n}ski et al.} [\cite{Dere}] described a derivation of Painlev\'{e} equations from Heun class equations and showed a direct relationship between deformed Heun class equations and all Painlev\'{e} equations. An easy way to understand the relation between the type of Painlev\'{e} equation and the singularity type of isomonodromic deformation {\color{black}was shown by Ohyama and Okumura in} [\cite{Ohyama}].

Motivated by [\cite{Dere}], using the method of isomonodromic deformations, the main purpose of this work is to study {\color{black}how to transform the Heun equation }into Painlev\'{e} equations. More precisely, we will show that the double-confluent Heun equation, through a proper transformation, {\color{black}translates} into the Painlev\'{e} {\color{black}III$'$}, biconfluent Heun equations relevant to Painlev\'{e} IV, and the general Heun equation corresponds to Painlev\'{e} VI. The research procedure of this paper can be illustrated in the Figure 1.

\begin{figure}[H]
    \centering
    \includegraphics[width=0.54\textwidth]{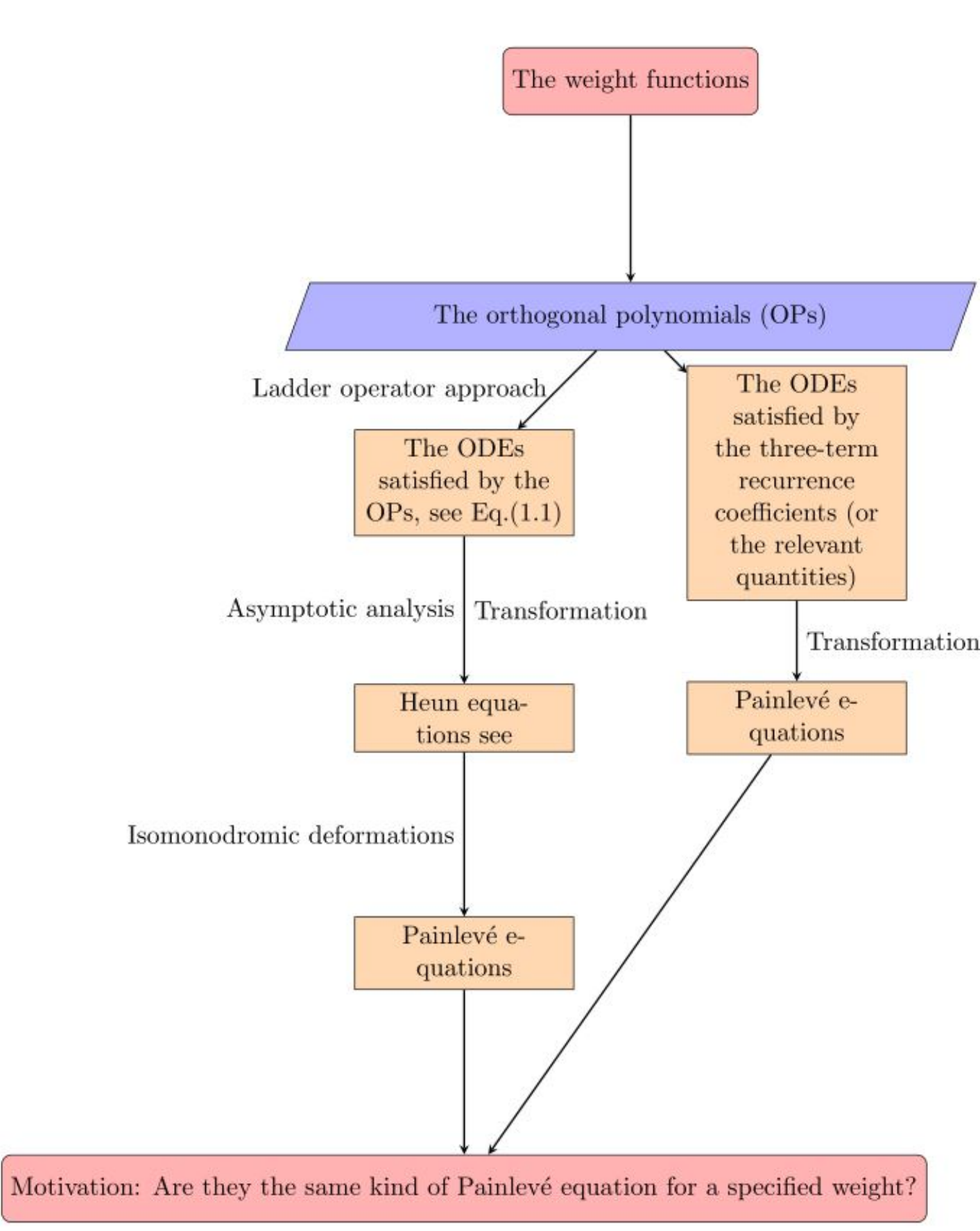}
    \caption{The research procedure of this paper}
    \label{figure}
\end{figure}

The rest of this paper is organized as follows. In Sec.2, we review the theory of the orthogonal polynomials $P_{n}(x)$ and the theory of isomonodromic deformations of linear second order differential equations. In Sec. 3, we applied the method of isomonodromic deformations to transform the double-confluent Heun equation into Painlev\'{e} {\color{black}III$'$}. In Sec.4 and 5, we devote to transform two particular bi-confluent Heun equations, associated with the polynomials $P_{n}(x)$ orthogonal with a Freud-type weight and a generalized Gaussian weight respectively, to Painlev\'{e} IV. In Sec.6, we find the relationships between the genal confluent Heun equation, with respect to the polynomials orthogonal with the Jacobi weight multiplied by $1-\chi_{(-a,a)}(x)$, and Painlev\'{e} VI. In Sec.7, under suitable double scaling, we derive the asymptotics of $D_{n}(t)$ in three cases of $s$ and $t$: $s\rightarrow\infty$ with $t$ fixed; $s\rightarrow0$ with $t>0$ fixed; $t\rightarrow\infty$ with $s>0$ fixed. The conclusions and the outline of some future ideas are shown in Sec.8.
\section{Preliminaries}
{\color{black}In this section, we state some elementary facts about the orthogonal polynomials at first.}

The monic polynomials $\left\{P_n(x)\right\}$ orthogonal with respect to a weight function $w(x)$ support on {\color{black}$\mathrm{L}(\subset\mathbb{R})$} are defined by
\begin{equation}\label{1.1}
\int_{\mathrm{{\color{black}L}}}P_{i}(x)P_{j}(x)w(x)dx=h_i\delta_{ij},
\end{equation}
where $h_i>0$ is the square of the $L^2-$norm of $i$th degree polynomials $P_i(x)$, and $\delta_{ij}$ denotes the Kronecker delta. The three-term recurrence relation of the monic orthogonal polynomials {\color{black}[\cite{a7},\cite{a6}] satisfies} that
\begin{equation}\label{1.3}
xP_n(x)=P_{n+1}(x)+\alpha_nP_{n}(x)+\beta_nP_{n-1}(x),
\end{equation}
with the initial conditions $P_{-1}(x)=0$ and $P_{0}(x)=1$. Here $\alpha_n$ and $\beta_n$ are known as the recurrence coefficients.

$D_n$ is the determinant of the Hankel matrix (also called moment matrix) generated by $w(x)$, which can be viewed as the partition function for the random matrix ensemble, i.e.
\begin{equation*}
D_n[w]=\det\left(\mu_{i+j}\right)_{i,j=0}^{n-1}=\det\left(\int_{{\color{black}\mathrm{L}}}x^{i+j}w(x)dx\right)_{i,j=0}^{n-1}=
\frac{1}{n!}\int_{\mathrm{{\color{black}L}}}\ldots\int_{{\color{black}\mathrm{L}}}\prod_{1\leq i<j\leq n}\left(x_i-x_j\right)^2\prod_{l=1}^n w(x_l)dx_l,
\end{equation*}
or
$$
\begin{aligned}
D_{n}[w]&=\operatorname{det}\left(\mu_{j+k}\right)_{j, k=0}^{n-1} =\left|\begin{array}{cccc}
\mu_{0} & \mu_{1}(t) & \cdots & \mu_{n-1} \\
\mu_{1} & \mu_{2}(t) & \cdots & \mu_{n} \\
\vdots & \vdots & & \vdots \\
\mu_{n-1} & \mu_{n} & \cdots & \mu_{2 n-2}
\end{array}\right|, \quad n \geqslant 1
\end{aligned}
$$
with initial conditions: {\color{black}$D_{-1}=0, D_{0}=1$},
where the elements of the Hankel matrix are the so-called moments
\begin{equation*}
\begin{aligned}
\mu_{j+k}:&=\int_{{\color{black}\mathrm{L}}}x^{j+k}w(x)dx.
\end{aligned}
\end{equation*}

According to the general theory of orthogonal polynomials of one variable[\cite{a19}], $D_{n}$ admits one more alternative representation
\begin{equation}\label{1.8}
\begin{aligned}
D_{n}=\prod^{n-1}_{j=0} h_{j},
\end{aligned}
\end{equation}
where $h_{j}$ is defined from the orthogonality in \eqref{1.1}.

Multiplying both sides of \eqref{1.3} by {\color{black}$P_{n-1}(x)w(x)$} and integrating this with respect to $x$ on {\color{black}$\mathrm{L}$}, the orthogonality
condition \eqref{1.1} gives
\begin{equation}\label{1.5}
\begin{aligned}
\beta_{n}=\frac{1}{h_{n-1}}\int_{{\color{black}\mathrm{L}}}xP_{n-1}(x)P_{n}(x)w(x)dx=\frac{h_{n}}{h_{n-1}}.
\end{aligned}
\end{equation}
Combining \eqref{1.8} with \eqref{1.5}, we have
\begin{equation*}
\begin{aligned}
\beta_{n}=\frac{D_{n-1}D_{n+1}}{D^2_{n}}.
\end{aligned}
\end{equation*}

The Heun class equations can be rewritten as {\color{black}[\cite{Dere}, Eq.(1.1)]}
\begin{equation}\label{H1}
\begin{aligned}
\big(\sigma(x)\partial_x^2+\tau(x)\partial_x+\eta(x)\big)\psi(x)=0,
\end{aligned}
\end{equation}
where $\sigma(x)\neq0$ is a polynomial of degree $\leq3, \tau$ of degree $\leq2$ and $\sigma\eta\leq4$. Assuming $\lambda, ~\mu$ are additional parameters, letting $\sigma, \tau, \eta$ be analytic at $\lambda$ and $\sigma(\lambda)\neq0$. We introduce the deformed form of \eqref{H1} as following {\color{black}[\cite{Dere}, Eq.(1.2)]}
\begin{equation}\label{t5}
\begin{aligned}
\bigg(\sigma(x)\partial^{2}_{x}+\bigg(\tau(x)-\frac{\sigma(x)}{x-\lambda}\bigg)\partial_{x}+\eta(x)-\eta(\lambda)-
\mu\big(\tau(\lambda)-\sigma'(\lambda)\big)-\mu^{2}\sigma(\lambda)+\frac{\mu\sigma(\lambda)}{x-\lambda}\bigg)\varphi(x)=0,
\end{aligned}
\end{equation}
where $\lambda$ is the position of the additional non-logarithmic singularity. We see that all the finite singularities of \eqref{H1} are the same as the equation above.

To prepare for later development, we recall for the reader that the theory of isomonodromic deformations [\cite{Iwas},\cite{Ohyama},\cite{Okamoto}] {\color{black}of} linear second order differential equations. {\color{black}We will use {\color{black}the} similar notation in [\cite{Dere}]}.

Assuming {\color{black}rational functions $\sigma, \tau, \eta$ depend on a time parameter $t$, we denote
 \begin{equation*}
\begin{aligned}
&p(x):=\frac{\tau(x)}{\sigma(x)}-\frac{1}{x-\lambda}, \\
&q(x):=\frac{1}{\sigma(x)}\bigg(\eta(x)-\eta(\lambda)-
\mu\big(\tau(\lambda)-\sigma'(\lambda)\big)-\mu^{2}\sigma(x)+\frac{\mu\sigma(\lambda)}{x-\lambda}\bigg).
\end{aligned}
\end{equation*}
Supposing that $t$ can be singled out among the parameters of $p(x),~q(x)$.
For convenience, we will notate $\frac{\partial\varphi}{\partial x}$ as $\varphi'$ and $\frac{\partial\varphi}{\partial t}$ as $\dot{\varphi}$. \eqref{t5} becomes a linear second order differential equation
\begin{equation}\label{t1}
\varphi''(x,t)+p(x)\varphi'(x,t)+q(x)\varphi(x,t)=0.
\end{equation}
Assuming that the certain solution $\varphi$ satisfies the following condition:
\begin{equation}\label{t11}
\dot{\varphi}(x,t)=a(x,t)\varphi'(x,t)+b(x,t)\varphi(x,t).
\end{equation}
}
After direct calculation, the so-called compatibility conditions of \eqref{t1} and \eqref{t11} are
\begin{equation}\label{t2}
\dot{p}-ap'+2b'-pa'+a''=0,
\end{equation}
\begin{equation}\label{t3}
\dot{q}+pb'-2qa'-q'a+b''=0,
\end{equation}
which control the isomonodromic deformation method.

The completely integrable Hamiltonian system depending on a variable $t$ (often called the 'time') is equivalent to the Painlev\'{e} equations, i.e.
Painlev\'{e} equations can be obtained from the non-autonomous Hamiltonian systems [\cite{Malmquist}]
\begin{equation*}
\frac{d\lambda}{dt}=\frac{\partial H}{\partial \mu},\quad \frac{d\mu}{dt}=-\frac{\partial H}{\partial \lambda},
\end{equation*}
{\color{black}by eliminating momentum $\mu$ from the Hamitonian systems}, where $H(t,\lambda,\mu)$ is called Painlev\'{e} Hamiltonians.
\section{Singularly perturbed Gaussian polynomials}
In this section, we will discuss certain characteristics of polynomials orthogonal with a singularly perturbed Gaussian weight {\color{black}[\cite{YLZC}]}
\begin{equation}\label{zhu2.1}
w_{\rm SPG}(x;t,\alpha)=|x|^{\alpha}{\color{black}\exp\left(-x^2-t/x^2\right)},~x\in\mathbb{R},~t,\alpha>0.
\end{equation}
Because the weight function is even, the recurrence coefficient $\alpha_n=0$ in \eqref{1.3}. Then \eqref{1.3} becomes
\begin{equation*}
xP_n(x)=P_{n+1}(x)+\beta_nP_{n-1}(x).
\end{equation*}
The moments {\color{black}are}
\begin{equation*}
\begin{aligned}
\mu_{j+k}=\left(1+(-1)^{j+k}\right)x^{\frac{\alpha+j+k+1}{4}}
K_{\frac{\alpha+j+k+1}{2}}\left(2\sqrt{x}\right).
\end{aligned}
\end{equation*}
Here, $K_{\nu}(x)$ is the modified Bessel function of the second kind [\cite[Eq. 10.32.10]{Olver}].

The following Lemma \ref{1} and \ref{2} have been proved by Derezi\'{n}ski {\color{black}et al.} [\cite{Dere}] using the methods of isomonodromic deformations.

For convenience, we will not display the $t$ and $x$ dependence in $ \sigma, \tau, \eta$, unless it is needed.
\begin{lem}\label{1}  {\color{black}[\cite{Dere}, Theorem 4.2, Case B]}
Assume that $\operatorname{deg} \sigma \leq 2$ and deg $\sigma \eta \leq 3$. Suppose that $\tau, \eta$, but not $\sigma$ depend on $t$. Let $m$ be a function of $t$ satisfying
\begin{equation}\label{t21}
\frac{\dot{\tau}(x)}{\sigma(x)}=m \frac{\tau^{\prime \prime}}{2},
\end{equation}
\begin{equation}\label{t22}
\dot{\eta}(x)-\dot{\eta}(\lambda)=m \frac{(\sigma \eta)^{'''}}{6}(x-\lambda).
\end{equation}
Define the compatibility functions
$$
a(t, \lambda, x)=\frac{m \sigma(x)}{x-\lambda}, \quad b(t, \lambda, \mu, x)=-\frac{m \sigma(\lambda) \mu}{x-\lambda},
$$
and the Hamiltonian
$$
H(t, \lambda, \mu):=m\left(\eta(\lambda)+\left(\tau(\lambda)-\sigma^{\prime}(\lambda)\right) \mu+\sigma(\lambda) \mu^{2}\right).
$$
Then $\lambda, \mu$ satisfy the Hamilton equations with respect to $H$, that is
$$
\begin{aligned}
\frac{\mathrm{d} \lambda}{\mathrm{d} t} &=\frac{\partial H}{\partial \mu}(t, \lambda, \mu),\\
\frac{\mathrm{d} \mu}{\mathrm{d} t} &=-\frac{\partial H}{\partial \lambda}(t, \lambda, \mu),
\end{aligned}
$$
if and only if \eqref{t2} and \eqref{t3} hold.
\end{lem}

\begin{lem}\label{2}{\color{black}[\cite{Dere}, Theorem 4.2, Case A]}
Assume that $s \in \mathbb{C}$ and $\sigma(s)=0$, so that we can write $\sigma(x)=(x-s) \rho(x)$ for a polynomial $\rho$ of degree $\leq 2$. We assume that $\sigma, \tau, \eta$ depend on $t. $ Let $m$ be a function of $t$ satisfying the following conditions
\begin{equation}\label{T1}
\partial_{t} \frac{\tau}{\sigma}(x)=\frac{m \tau(s)}{(x-s)^{2}},
\end{equation}
\begin{equation}\label{T2}
\begin{aligned}
\frac{\dot{\sigma}}{\sigma}(x)(\eta(x)-\eta(\lambda)) &-\dot{\eta}(x)+\dot{\eta}(\lambda)=m\left(\frac{(\eta(x)-\eta(\lambda))(\lambda-s) \rho(x)}{(x-\lambda)(x-s)}-\eta^{\prime}(\lambda) \rho(\lambda)\right) \\
&+\frac{(\lambda-s)}{(x-\lambda)^{2}}\left(2 \rho \eta(x)-2 \rho \eta(\lambda)-\left((\rho \eta)^{\prime}(x)+(\rho \eta)^{\prime}(\lambda)\right)(x-\lambda)\right),
\end{aligned}
\end{equation}
\begin{equation}\label{T3}
\begin{aligned}
\frac{\dot{\sigma}}{\sigma}(x)-\frac{\dot{\sigma}}{\sigma}(\lambda)=m \rho(s) & \frac{(x-\lambda)}{(x-s)(\lambda-s)} .
\end{aligned}
\end{equation}
Define the compatibility functions
$$
a(t, \lambda, x)=\frac{m(\lambda-s) \rho(x)}{x-\lambda}, \quad b(t, \lambda, \mu, x)=-\frac{m(\lambda-s) \rho(\lambda) \mu}{x-\lambda},
$$
and the Hamiltonian
$$
H(t, \lambda, \mu)=m\left(\eta(\lambda)+\left(\tau(\lambda)-(\lambda-s) \rho^{\prime}(\lambda)\right) \mu+\sigma(\lambda) \mu^{2}\right).
$$
Then $\lambda, \mu$ satisfy the Hamilton equations with respect to $H$ if and only if \eqref{t2} and \eqref{t3} hold.
\end{lem}

From the work of Yu et al. [\cite{YLZC}],
one can find that the monic polynomials $P_n(x)$ orthogonal with the weight \eqref{zhu2.1} satisfy a second order differential equation
\begin{equation}\label{heun0}
P_n''(x)+T_n(x)P_n'(x)+Q_n(x)P_n(x)=0,
\end{equation}
where $T_n(x)$ and $Q_n(x)$ have the following expansions:
\begin{subequations}
\begin{align*}
T_{n}(x)=&\frac{2(2\beta_{n}+2\beta_{n+1}-\alpha-2n-1)}{x(2x^2+2\beta_{n}+2\beta_{n+1}-\alpha-2n-1)}
-2x+\frac{2t}{x^3}+\frac{\alpha}{x},\\
Q_{n}(x)=&-\frac{2\beta_{n}-n}{x^2}-\frac{3(1-(-1)^n)t}{x^4}+\frac{2(2\beta_{n}+2\beta_{n+1}-\alpha-2n-1)\big((2\beta_{n}-n)x^2+(1-(-1)^n)t\big)}{x^4(2x^2
+2\beta_{n}+2\beta_{n+1}-\alpha-2n-1)}\notag\\
&-\big(2x-\frac{(1+(-1)^n)t}{x^3}+\frac{2\beta_{n}-n-\alpha}{x}\big)\frac{(2\beta_{n}-n)x^2+(1-(-1)^n)t}{x^3}8\\
&+\frac{\beta_{n}(2x^2+2\beta_{n}+2\beta_{n+1}-\alpha-2n-1)(2x^2+2\beta_{n-1}+2\beta_{n}-\alpha-2n+1)}{x^4}.\notag
\end{align*}
\end{subequations}
According to the asymptotic result in [35],
\begin{equation*}
\beta_{n}=\dfrac{2n+\alpha}{4}+\mathcal{O}(1),\quad n\rightarrow\infty.
\end{equation*}
Consequently, $T_{n}(x)$ and $Q_{n}(x)$ can be given by
\begin{equation*}
T_{n}(x)=-2x+\frac{2t}{x^3}+\frac{\alpha}{x}+\mathcal{O}(1), \quad n\rightarrow\infty,
\end{equation*}
\begin{equation*}
Q_{n}(x)=2n+\alpha+\mathcal{O}(1), \quad n\rightarrow\infty.
\end{equation*}
Omitting the error terms, \eqref{heun0} becomes the equation below:
\begin{equation}\label{6.14}
\begin{aligned}
\widehat{P}''_{n}(x)+(-2x+\frac{2t}{x^3}+\frac{\alpha}{x})\widehat{P}_{n}'(x)+(2n+\alpha)\widehat{P}_{n}(x)=0.
\end{aligned}
\end{equation}

Through a transformation, \eqref{6.14} satisfies a double-confluent Heun equation with parameters appearing in the weight.
\begin{lem} (See [\cite{YLZC}]) Let $n\rightarrow\infty, t\rightarrow0^{+}, \kappa:=t(2n+\alpha)$ is fixed. Then $u(x):=\widehat{P}_{n}(2^{-\frac{1}{4}}t^{\frac{1}{2}}x^{\frac{1}{2}})$ satisfies the double-confluent Heun equation [\cite{IMHI}]
\begin{equation}\label{2.6}
\begin{aligned}
u''(x)+\left(\dfrac{\sqrt{2}}{x^{2}}+\dfrac{1+\alpha}{2x}-\frac{\sqrt{2}t}{2}\right)u'(x)+\dfrac {\sqrt{2}t(2n+\alpha)}{8x}u(x)=0.
\end{aligned}
\end{equation}
\end{lem}

\begin{thm}\label{1}
\eqref{2.6} can be {\color{black}transformed} into a Painlev\'{e} III$'$.
\end{thm}
\begin{proof}
\eqref{2.6} can be rewritten as
\begin{equation*}
\begin{aligned}
x^{2}u''(x)+\left(\sqrt{2}+\dfrac{1+\alpha}{2}x-\frac{\sqrt{2}t}{2}x^{2}\right)u'(x)+\dfrac {\sqrt{2}t(2n+\alpha)x}{8}u(x)=0.
\end{aligned}
\end{equation*}
Thus, we have
\begin{equation*}
\begin{aligned}
\sigma(x)=x^{2},\quad \tau(x)=\sqrt{2}+\dfrac{1+\alpha}{2}x-\frac{\sqrt{2}t}{2}x^{2},\quad\eta(x)=\dfrac {\sqrt{2}t(2n+\alpha)x}{8}.
\end{aligned}
\end{equation*}
Because $\operatorname{deg} \sigma \leq 2$ and deg $\sigma \eta \leq 3$,
we set $c(x)=m\sigma(x)$ and if and only if $m=\frac{1}{t}$, \eqref{t21} and \eqref{t22} are hold in the Lemma \ref{1}.

Define the compatibility functions
$$
a(t, \lambda, x)=\frac{x^2}{t(x-\lambda)}, \quad b(t, \lambda, \mu, x)=-\frac{\lambda^2 \mu}{t(x-\lambda)},
$$
and the Hamiltonian
$$
tH(t, \lambda, \mu):=\lambda^2\mu^{2}+\left(\sqrt{2}+\dfrac{1+\alpha}{2}\lambda-\frac{\sqrt{2}t}{2}\lambda^{2}-2\lambda\right) \mu+\dfrac {\sqrt{2}t(2n+\alpha)\lambda}{8}.
$$
Then we have
\begin{equation}\label{2.7}
\begin{aligned}
\frac{\mathrm{d} \lambda}{\mathrm{d} t} &=\frac{\partial H}{\partial \mu}(t, \lambda, \mu)=\frac{1}{t}\left(2\lambda^2\mu+\sqrt{2}+\dfrac{1+\alpha}{2}\lambda-\frac{\sqrt{2}t}{2}\lambda^{2}-2\lambda\right),
\end{aligned}
\end{equation}
\begin{equation}\label{2.8}
\begin{aligned}
\frac{\mathrm{d} \mu}{\mathrm{d} t} &=-\frac{\partial H}{\partial \lambda}(t, \lambda, \mu)=-\frac{1}{t}\left(2\lambda \mu^{2}+\big(\dfrac{1+\alpha}{2}-\sqrt{2}t\lambda-2\big) \mu+\dfrac {\sqrt{2}t(2n+\alpha)}{8}\right).
\end{aligned}
\end{equation}
{\color{black}
Solving for $\mu$ in \eqref{2.7} gives
\begin{equation}\label{2.9}
\begin{aligned}
\mu &=\frac{1}{2\lambda^2}\left(t\frac{\mathrm{d} \lambda}{\mathrm{d} t}-\sqrt{2}-\dfrac{1+\alpha}{2}\lambda+\frac{\sqrt{2}t}{2}\lambda^{2}+2\lambda\right).
\end{aligned}
\end{equation}

Substituting \eqref{2.9} into \eqref{2.8}, we obtain
\begin{equation}\label{2.12}
\begin{aligned}
\frac{\mathrm{d^{2}} \lambda}{\mathrm{d} t^{2}} =\frac{1}{\lambda}\left(\frac{d\lambda}{dt}\right)^2-\frac{1}{t}\frac{d\lambda}{dt}+\frac{\sqrt{2}(1-2n-2\alpha)}{4t}\lambda^2+\frac{\lambda^3}{2}
+\frac{\sqrt{2}(3-\alpha)}{2t^2}-\frac{2}{\lambda t^2}.
\end{aligned}
\end{equation}}
Let $\lambda=\frac{y}{t}$, \eqref{2.12} becomes
\begin{equation}\label{2.13}
\begin{aligned}
\frac{\mathrm{d^{2}} y}{\mathrm{d} t^{2}} =\frac{1}{y}\left(\frac{dy}{dt}\right)^2-\frac{1}{t}\frac{dy}{dt}+\frac{\sqrt{2}(1-2n-2\alpha)y^2}{4t^2}+\frac{y^3}{2t^2}
+\frac{\sqrt{2}(3-\alpha)}{2t}-\frac{2}{y},
\end{aligned}
\end{equation}
which is a particular Painlev\'{e} III[\cite{okamoto1}], i.e. P$_{\mathrm{III}'}(\sqrt{2}(1-2n-2\alpha),2\sqrt{2}(3-\alpha),2,-8)$.

\end{proof}

\begin{rem}
Based on Theorem \ref{1} and the results in [\cite{YLZC}, Thm. 3.3], we can see that although the quantities $a_n$ related to the three-term recurrence coefficients, $a_n=2\beta_n+2\beta_{n+1}-(\alpha+2n+1)$, and the quantities $y$ in (\ref{2.13})  associated with orthogonal polynomials, which satisfy Painlev\'{e} equations, differ in specific parameters, they all belong to the same class of Painlev\'{e} equations, i.e. P$_{\mathrm{III}'}$. This is what we want to see.
\end{rem}

\section{{\color{black}The deformed Freud polynomials}}
In [\cite{Clarkson}], {\color{black}Clarkson and Jordaan} analyzed the asymptotic behavior of the sequences of monic polynomials that are orthogonal with respect to the deformed Freud weight
\begin{equation}\label{df0}
w_{{\color{black}\mathrm{DF}}}(x,\alpha,t)=|x|^{\alpha}\exp\left({-x^{4}+tx^{2}}\right),\quad x\in\mathbb{R},~t>0,~\alpha>0.
\end{equation}
Clarkson and Jordaan obtained a particular bi-confluent Heun equation [\cite{ClarksonJAT}, Eq.(4.14)] based on the the asymptotic behavior of the recurrence coefficients $\beta_{n}$. In this section, we will try to turn this bi-confluent Heun equation into Painlev\'{e} IV.

Let's first review a result about the differential equation satisfied by orthogonal polynomials associated with the weight (\ref{df0}).

\begin{thm}(See [\cite{Clarkson}])
 {\color{black}The} monic polynomials $P_{n}(x ; t)$ {\color{black}orthogonal} with respect to generalized Freud weight \eqref{df0} satisfy the differential equation
\begin{equation*}
\frac{\mathrm{d}^{2} \mathbb{P}_{n}}{\mathrm{~d} x^{2}}(x ; t)+\mathbb{T}_{n}(x ; t) \frac{\mathrm{d} \mathbb{P}_{n}}{\mathrm{~d} x}(x ; t)+\mathbb{Q}_{n}(x ; t) \mathbb{P}_{n}(x ; t)=0,
\end{equation*}
where
\begin{subequations}
\begin{align}
\mathbb{T}_{n}(x ; t)&=-4 x^{3}+2 t x+\frac{2 \alpha+1}{x}-\frac{2 x}{x^{2}-\frac{1}{2} t+\beta_{n}+\beta_{n+1}},\label{df1}\\
\mathbb{Q}_{n}(x ; t)&=4 n x^{2}+4 \beta_{n}+16 \beta_{n}\left(\beta_{n}+\beta_{n+1}-\frac{1}{2} t\right)\left(\beta_{n}+\beta_{n-1}-\frac{1}{2} t\right)
+4(2 \alpha+1)(-1)^{n} \beta_{n}\label{df2}\\
&-\frac{8 \beta_{n} x^{2}+(2 \alpha+1)\left[1-(-1)^{n}\right]}{x^{2}-\frac{1}{2} t+\beta_{n}+\beta_{n+1}}+(2 \alpha+1)\left[1-(-1)^{n}\right]\left(t-\frac{1}{2 x^{2}}\right)\notag .
\end{align}
\end{subequations}
\end{thm}

The asymptotic expansion of $\beta_n(t; \alpha)$ in (\ref{df1}) and (\ref{df2}) as $n \rightarrow\infty$, for $t, \alpha \in \mathbb{R}$ was
studied by Clarkson and Jordaan in [\cite{ClarksonJAT}]. They found $\beta_{n}=$ $\frac{1}{6} \sqrt{3 n}+\mathcal{O}(1)$ as $n \rightarrow \infty$. So it follows from (\ref{df1}) and (\ref{df2}) that
\begin{equation*}
\begin{aligned}
&\mathbb{T}_{n}(x ; t)=-4 x^{3}+2 t x+\frac{2 \alpha+1}{x}+\mathcal{O}\left(n^{-1 / 2}\right) ,\\
&\mathbb{Q}_{n}(x ; t)=\left(\frac{4}{3} n\right)^{3 / 2}+\mathcal{O}(n).
\end{aligned}
\end{equation*}

To obtain the Heun equation, using the same method {\color{black}in} obtaining (\ref{6.14}), let us to consider the equation,
$$
\frac{\mathrm{d}^{2} \widehat{\mathbb{P}}_{n}(x ; t)}{\mathrm{~d} x^{2}}-\left(4 x^{3}-2 t x-\frac{2 \alpha+1}{x}\right) \frac{\mathrm{d} \widehat{\mathbb{P}}_{n}(x ; t)}{\mathrm{~d} x}+\left(\frac{4}{3} n\right)^{3 / 2} \widehat{\mathbb{P}}_{n}(x ; t)=0.
$$

\begin{lem}(See [\cite{ClarksonJAT}]) Let $n\rightarrow\infty$, then $u(x):=\widehat{\mathbb{P}}_{n}(2^{\frac{1}{2}}x^{\frac{1}{2}})$ satisfies the bi-confluent Heun equation [\cite{Olver}]
\begin{equation}\label{2.14}
\begin{aligned}
u''(x)+\left(\dfrac{1+\alpha}{x}+\dfrac{\sqrt{2}t}{2}-x\right)u'(x)+\dfrac {\sqrt{6}n^{\frac{3}{2}}}{9x}u(x)=0.
\end{aligned}
\end{equation}
\end{lem}

\begin{thm}\label{22}
\eqref{2.14} can be transformed into a Painlev\'{e} IV.
\end{thm}

\begin{proof}
\eqref{2.14} can be rewritten as
\begin{equation*}
\begin{aligned}
xu''(x)+\left(1+\alpha+\dfrac{\sqrt{2}t}{2}x-x^2\right)u'(x)+\dfrac {\sqrt{6}n^{\frac{3}{2}}}{9}u(x)=0.
\end{aligned}
\end{equation*}
Thus, we have
\begin{equation*}
\begin{aligned}
\sigma(x)=x,\quad \tau(x)=1+\alpha+\dfrac{\sqrt{2}t}{2}x-x^2,\quad\eta(x)=\dfrac {\sqrt{6}n^{\frac{3}{2}}}{9}.
\end{aligned}
\end{equation*}
Because $\operatorname{deg} \sigma \leq 2$ and deg $\sigma \eta \leq 3$,
we set $c(x)=m\sigma(x)$, if and only if $m=-\frac{\sqrt{2}}{2}$, $\sigma(x), \tau(x), \eta(x)$ satisfy \eqref{t21} and \eqref{t22} in the Lemma \ref{1}.

Define the compatibility functions
$$
a(t, \lambda, x)=-\frac{\sqrt{2}x}{2(x-\lambda)}, \quad b(t, \lambda, \mu, x)=\frac{\sqrt{2}\lambda\mu}{2(x-\lambda)},
$$
and the Hamiltonian
$$
H(t, \lambda, \mu):=-\frac{\sqrt{2}}{2}\left(\lambda\mu^{2}+\big(\alpha+\frac{\sqrt{2}t}{2}\lambda-\lambda^2\big) \mu+\dfrac {\sqrt{6}n^{\frac{3}{2}}}{9}\right).
$$
Then we have
\begin{equation}\label{2.15}
\begin{aligned}
\frac{\mathrm{d} \lambda}{\mathrm{d} t} &=\frac{\partial H}{\partial \mu}(t, \lambda, \mu)=-\sqrt{2}\lambda\mu-\dfrac{\sqrt{2}\alpha}{2}-\frac{t}{2}\lambda+\frac{\sqrt{2}}{2}\lambda^2,
\end{aligned}
\end{equation}
\begin{equation}\label{2.16}
\begin{aligned}
\frac{\mathrm{d} \mu}{\mathrm{d} t} &=-\frac{\partial H}{\partial \lambda}(t, \lambda, \mu)=\frac{\sqrt{2}}{2}\left( \mu^{2}+\big(\frac{\sqrt{2}}{2}t-2\lambda\big) \mu\right).
\end{aligned}
\end{equation}
{\color{black}
Solving $\mu$ from \eqref{2.15}, we obtain
\begin{equation}\label{2.19}
\begin{aligned}
\mu &=-\frac{1}{\sqrt{2}\lambda}\frac{\mathrm{d} \lambda}{\mathrm{d} t}-\dfrac{\alpha}{2\lambda}-\frac{t}{2\sqrt{2}}+\frac{1}{2}\lambda.
\end{aligned}
\end{equation}

Substituting \eqref{2.19} into \eqref{2.16}, we obtain
\begin{equation}\label{2.20}
\begin{aligned}
\frac{\mathrm{d^{2}} \lambda}{\mathrm{d} t^{2}} =\frac{1}{2\lambda}\left(\frac{d\lambda}{dt}\right)^2+\frac{3\lambda^3}{4}-\frac{\sqrt{2}}{2}t\lambda^{2}
+\frac{t^{2}\lambda}{8}-\frac{\alpha+1}{2}\lambda-\frac{\alpha^2}{4\lambda}.
\end{aligned}
\end{equation}}

We introduce the new independent variables $\lambda=-\frac{\sqrt{2}y}{2}, t=2x$, \eqref{2.20} becomes
\begin{equation*}
\frac{\mathrm{d^{2}} y}{\mathrm{d} x^{2}} =\frac{1}{2y}\left(\frac{dy}{dx}\right)^2+\frac{3y^3}{2}+4xy^{2}
+2\big(x^{2}-(\alpha+1)\big)y-\frac{2\alpha^2}{y},
\end{equation*}
which is a Painlev\'{e} IV.
\end{proof}
\begin{rem}
As shown in [\cite{Clarkson}, Lem. 4.1], the three-term recurrence coefficients $\beta_{n}(x,\alpha)$ also satisfy a Painlev\'{e} IV, which is only different in terms of parameters with the above equation.
\end{rem}

\section{The Gaussian weight with a single jump}
Chen and Pruessner [\cite{ChenP}] studied the Gaussian weight with a single jump
\begin{align}\label{4.0}
w_{{\color{black}\mathrm{GJ}}}(x,t)={\color{black}\mathrm{exp}(-x^{2})}(A+B\theta(x-t)),\quad {\color{black}~x, t\in\mathbb{R}},
\end{align}
where $\theta(x)$ is the Heaviside step function, i.e., $\theta(x)$ is $1$ for $x > 0$ and $0$ otherwise, and with parameters $A=1-\frac{\beta}{2}, B=\beta$, the
real $\beta$ representatives the height of the jump. They obtained a particular Painlev\'{e} IV for the diagonal recurrence coefficient $\alpha_{n}(t)$ of the monic orthogonal polynomials with respect to \eqref{4.0}. For the constants $A, B$ with $A \geq 0, A + B \geq 0$,  Min and Chen [\cite{MC}] derived the second order differential equation satisfied by the auxiliary quantity $R_n(t)$, which is related to the Painlev\'{e} IV. They also derived that the the second order differential equations satisfied by the orthogonal polynomials $\mathcal{P}_{n}(x)$ associated with a particular bi-confluent Heun equation as $n$ tends to infinity.

In this section, we will display the bi-confluent Heun equation can be transformed into a Painlev\'{e} equation. As shown in [\cite{MC}], the monic orthogonal polynomials $\mathcal{P}_{n}(x)$ satisfy the second order differential equation
\begin{equation*}
\mathcal{P}_{n}^{\prime \prime}(x)+\mathcal{T}_n(x)\mathcal{P}_{n}^{\prime}(x)+\mathcal{Q}_n(x)\mathcal{P}_{n}(x)=0,
\end{equation*}
where
\begin{equation*}
\begin{aligned}
\mathcal{T}_n(x)=&\frac{R_{n}(t)}{(x-t)(2 x-2t +R_{n}(t))}-2x,\\
\mathcal{Q}_n(x)=&2 n-\frac{R_{n}^{\prime}(t)-R_{n}^{2}(t)+2 t R_{n}(t)}{4(x-t)^{2}}+\frac{R_{n}(t)(R_{n}^{\prime}(t)-R_{n}^{2}(t)+2 t R_{n}(t))}{4(x-t)^{2}(2 x-2 t+R_{n}(t))}\\
&+\frac{(R_{n}^{\prime}(t))^{2}-R_{n}^{4}(t)+4 t R_{n}^{3}(t)+(8 n-4 t^{2}) R_{n}^{2}(t)}{8(x-t) R_{n}(t)}.
\end{aligned}
\end{equation*}
When $n$ is large, the auxiliary quantity $R_n$ {\color{black}has the following asymptotic behavior} [\cite{MC}]
\begin{equation*}
R_n(t)=\frac{2\sqrt{6n}}{3}+\frac{4t}{3}+\mathcal{O}(1),
\end{equation*}
Therefore, as $n\rightarrow\infty$, let us to consider the equation,
\begin{equation*}
\widehat{\mathcal{P}}_{n}^{\prime \prime}(x)+\bigg(\frac{1}{x-t}-2x\bigg)\widehat{\mathcal{P}}_{n}^{\prime}(x)+\frac{4\sqrt{6}n^{\frac{2}{3}}}{9(x-t)}\widehat{\mathcal{P}}_{n}(x)=0 .
\end{equation*}

\begin{lem}
 Let $n\rightarrow\infty$, then $u(x):=\widehat{\mathcal{P}}_{n}(\frac{x}{\sqrt{2}}+t)$ satisfies the bi-confluent Heun equation [\cite{Olver}]
\begin{equation}\label{2.21}
\begin{aligned}
u''(x)-\left(-\dfrac{1}{x}+\sqrt{2}t+x\right) u'(x)+\dfrac {4\sqrt{3}n^{\frac{3}{2}}}{9x}u(x)=0.
\end{aligned}
\end{equation}
\end{lem}

\begin{thm}\label{3}
\eqref{2.21} can be transformed into a Painlev\'{e} IV.
\end{thm}

\begin{proof}
\eqref{2.21} can be rewritten as
\begin{equation*}
\begin{aligned}
xu''(x)+\left(1-\sqrt{2}tx-x^2\right)u'(x)+\dfrac {4\sqrt{3}n^{\frac{3}{2}}}{9}u(x)=0.
\end{aligned}
\end{equation*}
Thus, we have
\begin{equation*}
\begin{aligned}
\sigma(x)=x,\quad \tau(x)=1-\sqrt{2}tx-x^2,\quad\eta(x)=\dfrac {4\sqrt{3}n^{\frac{3}{2}}}{9}.
\end{aligned}
\end{equation*}
Because $\operatorname{deg} \sigma \leq 2$ and deg $\sigma \eta \leq 3$,
we set $c(x)=m\sigma(x)$, if and only if $m=\sqrt{2}$, $\sigma(x), \tau(x), \eta(x)$ satisfy \eqref{t21} and \eqref{t22} in the Lemma \ref{1}.

Define the compatibility functions
$$
a(t, \lambda, x)=\frac{\sqrt{2}x}{x-\lambda}, \quad b(t, \lambda, \mu, x)=-\frac{\sqrt{2}\lambda\mu}{x-\lambda},
$$
and the Hamiltonian
$$
H(t, \lambda, \mu):=\sqrt{2}\lambda\mu^{2}-\big(2t\lambda+\sqrt{2}\lambda^2\big) \mu+\dfrac {4\sqrt{6}n^{\frac{3}{2}}}{9}.
$$
Then we have
\begin{equation}\label{2.22}
\begin{aligned}
\frac{\mathrm{d} \lambda}{\mathrm{d} t} &=\frac{\partial H}{\partial \mu}(t, \lambda, \mu)=2\sqrt{2}\lambda\mu-(2t\lambda+\sqrt{2}\lambda^2),
\end{aligned}
\end{equation}
\begin{equation}\label{2.23}
\begin{aligned}
\frac{\mathrm{d} \mu}{\mathrm{d} t} &=-\frac{\partial H}{\partial \lambda}(t, \lambda, \mu)=-\sqrt{2}\mu^{2}+(2t+2\sqrt{2}\lambda)\mu.
\end{aligned}
\end{equation}

{\color{black}Solving $\mu$ from \eqref{2.22}, we obtain
\begin{equation}\label{2.26}
\begin{aligned}
\mu &=\frac{1}{2\sqrt{2}\lambda}\frac{\mathrm{d} \lambda}{\mathrm{d} t}+\frac{t}{\sqrt{2}}+\frac{\lambda}{2}.
\end{aligned}
\end{equation}

Substituting \eqref{2.26} into \eqref{2.23}, we obtain
\begin{equation}\label{2.27}
\begin{aligned}
\frac{\mathrm{d^{2}} \lambda}{\mathrm{d} t^{2}} =\frac{1}{2\lambda}\left(\frac{d\lambda}{dt}\right)^2+3\lambda^3+4\sqrt{2}t\lambda^{2}
+2t^{2}\lambda-2\lambda.
\end{aligned}
\end{equation}
}

We introduce the new independent variables $\lambda=-\frac{y}{\sqrt{2}}$, \eqref{2.27} becomes
\begin{equation*}
\frac{\mathrm{d^{2}} y}{\mathrm{d} t^{2}} =\frac{1}{2y}\left(\frac{dy}{dt}\right)^2+\frac{3y^3}{2}+4ty^{2}
+2(t^{2}-1)y-2y,
\end{equation*}
which is a Painlev\'{e} IV.
\end{proof}

\begin{rem} Through appropriate transformations, Min and Chen [\cite{MC}] obtain that the auxiliary quantity $R_n(t)$ satisfied a Painlev\'{e} IV equation, which coincides with the Theorem 5.1, except for the parameters.
\end{rem}

\section{The Jacobi weight multiplied by $1-\chi_{(-a,a)}(x)$}
Using the ladder operator technique for orthogonal polynomials
and the associated supplementary conditions, Min and Chen [\cite{Min}] derived the second-order differential equations for two auxiliary quantities with respect to the weight
\begin{equation}\label{2.29}
 w_{{\color{black}\mathrm{{\color{black}JC}}}}(x,\alpha,a)=(1-x^2)^{\alpha}(1-\chi_{(-a,a)}(x)), x\in [-1,1], a\in(0,1),\alpha>0,
 \end{equation}
where $\chi_{(-a,a)}(x)$ is the characteristic(or the indicator) function of  the interval $(-a,a)$, namely $\chi_{(-a,a)}(x)=1$ if $x\in(-a,a)$ and $\chi_{(-a,a)}(x)=0$ if $x\notin(-a,a)$.

In [\cite{Zhan2}], the second order differential equation for polynomials $\mathbf{P}_n(x,\alpha,a)$ orthogonal with respect to the weight \eqref{2.29}
was obtained. They also used the asymptotic expressions for auxiliary quantities to reduce the second-order differential equation to a simpler
form, which can be transformed into one of the Heun equations.

In this section, we will show the Heun equation can be transformed into a Painlev\'{e} equation. {\color{black}According to the work of Min and Chen in [\cite{Min}], the monic orthogonal polynomials $\mathbf{P}_n(x,\alpha,a)$ satisfy

\begin{equation}\label{j0}
\mathbf{P}_{n}^{\prime \prime}(x)+\mathbf{T}_{n}(x, a) \mathbf{P}_{n}^{\prime}(x)+\mathbf{Q}_{n}(x, a) \mathbf{P}_{n}(x)=0,
\end{equation}
where
\begin{equation*}
\begin{aligned}
\mathbf{T}_{n}(x, a)=& \frac{2(\alpha+1) x}{x^{2}-1}+\frac{2 x}{x^{2}-a^{2}}-\frac{2 x(2 \alpha+2 n+1)}{\left(a-a^{3}\right) R_{n}(a)-\left(a^{2}-z^{2}\right)(2 \alpha+2 n+1)}, \\
\mathbf{Q}_{n}(x, a)=& \frac{\left(1-a^{2}\right)\left(\left(a^{2}+x^{2}\right)(2 \alpha+2 n+1)+a\left(a^{2}-1\right) R_{n}(a)\right)r_n(a)}{\left(x^{2}-1\right)\left(a^{2}-x^{2}\right)\left(\left(a^{2}-x^{2}\right)(2 \alpha+2 n+1)+a\left(a^{2}-1\right) R_{n}(a)\right)}\\
&-\frac{n\left(\left(a^{2}-x^{2}\right)^{2}(2 \alpha+2 n+1)+a\left(a^{2}-1\right) R_{n}(a)\left(a^{2}-3 x^{2}\right)\right)}{\left(x^{2}-1\right)\left(x^{2}-a^{2}\right)\left(\left(a-a^{3}\right) R_{n}(a)+\left(x^{2}-a^{2}\right)(2 \alpha+2 n+1)\right)} \\
&+\frac{\left(n^{2}+2 \alpha n\right)\left(a^{2}-x^{2}\right)+2(a^2-1)(\alpha+n)r_n(a)+(4((\alpha+n)^2-1)\beta_n-n(2\alpha+n)}{\left(x^{2}-a^{2}\right)\left(x^{2}-1\right)},
\end{aligned}
\end{equation*}
with $r_{n}(a)$ and $\beta_{n}(a)$ given by
$$
r_{n}(a)=\frac{a\left(-\left(a^{2}-1\right) R_{n}^{\prime}(a)+\left(a^{2}-1\right) R_{n}(a)^{2}+2 a(\alpha+n) R_{n}(a)\right)}{2\left(\left(a^{2}-1\right) R_{n}(a)+a(2 \alpha+2 n+1)\right)},
$$
and
$$
\beta_{n}(a)=\frac{1}{2 n+2 \alpha-1}\left[\frac{\left(r_{n}(a)+n\right)\left(r_{n}(a)+2 \alpha+n\right)}{a R_{n}(a)+2 \alpha+2 n+1}-\frac{a r_{n}(a)^{2}}{R_{n}(a)}\right].
$$
Hence, the coefficients of \eqref{j0} depend only on $R_{n}(a)$ and $R_{n}^{\prime}(a)$.

For large $n$ and $a>0$, we find from [\cite{Zhan2}] that
$$
R_{n}(a)=\frac{2 a n}{1-a^{2}}+\frac{2 \alpha a+a+1}{1-a^{2}}+\mathcal{O}(1).
$$
As $n\rightarrow\infty$, \eqref{j0} degenerates into a related equation as follows
\begin{equation*}
\widehat{\mathbf{P}}_{n}^{\prime \prime}(x)+\left(\frac{2 x}{x^{2}-a^{2}}+\frac{2(\alpha+1) x}{x^{2}-1}-\frac{2}{x}\right)\widehat{\mathbf{P}}_{n}^{\prime}(x)-\frac{n a+x^{2}(2 \alpha+n+1) n}{\left(x^{2}-1\right)\left(x^{2}-a^{2}\right)}\widehat{\mathbf{P}}_{n}(x)=0.
\end{equation*}

\begin{lem}
For $n\rightarrow\infty,$ let $u(x)=\widehat{\mathbf{P}}_{n}(x^{\frac{1}{2}})$ and $t=a^2$, we obtain a general Heun equation
\begin{equation}\label{3.22}
u''(x)+\left(-\dfrac{1}{2x}+\dfrac{1+\alpha}{x-1}+\frac{1}{x-t}\right)u'(x)-\dfrac {n(n+2\alpha+1)x+n\sqrt{t}}{4x(x-1)(x-t)}u(x)=0.
\end{equation}
\end{lem}

\begin{thm}\label{3}
\eqref{3.22} can be transformed into a Painlev\'{e} VI.
\end{thm}
\begin{proof}
\eqref{3.22} can be rewritten as
\begin{equation*}
\begin{aligned}
x(x-1)(x-t)u''(x)+\left(-\dfrac{1}{2}(x-1)(x-t)+(1+\alpha)x(x-t)+x(x-1)\right)u'(x)\\
-\dfrac {n(n+2\alpha+1)x+n\sqrt{t}}{4}u(x)=0.
\end{aligned}
\end{equation*}

Thus, we have
\begin{equation*}
\begin{aligned}
&\sigma(x)=x(x-1)(x-t),\quad \tau(x)=-\dfrac{1}{2}(x-1)(x-t)+(1+\alpha)x(x-t)+x(x-1),\\
&\eta(x)=-\dfrac {n(n+2\alpha+1)x+n\sqrt{t}}{4}.
\end{aligned}
\end{equation*}
Let $\sigma(x)=(x-t)\rho(x)$, where $\rho(x)=x(x-1)$ and deg $\rho=2$.
Setting $c(x)=m(\lambda-t)\rho(x)$, if and only if {\color{black}$m=\frac{1}{t(t-1)}$}, we have
\begin{equation*}
\begin{aligned}
\partial_{t} \frac{\tau}{\sigma}(x)-\frac{m \tau(t)}{(x-t)^{2}}&=\frac{1}{(x-t)^2}-\frac{m\tau(t)}{(x-t)^2}=\frac{1-mt(t-1)}{(x-t)^2}=0.
\end{aligned}
\end{equation*}
There is a simple identity {\color{black}which} we will use
\begin{equation*}
\begin{aligned}
&\frac{(\lambda-t)}{(x-\lambda)^{2}}(2 \rho \eta(x)-2 \rho \eta(\lambda)-\left((\rho \eta)^{\prime}(x)+(\rho \eta)^{\prime}(\lambda)\right)(x-\lambda)\\
&=-\frac{(\rho \eta(x))^{'''}}{6}(x-\lambda)(\lambda-t)=\dfrac {n(n+2\alpha+1)}{4}(x-\lambda)(\lambda-t).
\end{aligned}
\end{equation*}

Note that
\begin{equation*}
\begin{aligned}
&\frac{mx(x-1)(\lambda-t)}{x-t}-m\lambda(\lambda-1)=\frac{m(x-t+t)(x-1)(\lambda-t)}{x-t}-m\lambda(\lambda-1)\\
&=m(x-1)(\lambda-t)+\frac{mt(x-t+t-1)(\lambda-t)}{x-t}-m\lambda(\lambda-1)\\
&=m(x-1)(\lambda-t)+mt(\lambda-t)+\frac{mt(t-1)(\lambda-t)}{x-t}-m\lambda(\lambda-1)\\
&=m(x-1)(\lambda-t)+mt(\lambda-t)+\frac{mt(t-1)(\lambda-x+x-t)}{x-t}-m\lambda(\lambda-1)\\
&=m(x-1)(\lambda-t)+mt(\lambda-t)+\frac{mt(t-1)(\lambda-x)}{x-t}+mt(t-1)-m\lambda(\lambda-1)\\
&=-\frac{mt(t-1)(x-\lambda)}{x-t}+m(\lambda-t)(x-\lambda).
\end{aligned}
\end{equation*}
Substituting $\sigma(x), \tau(x), \eta(x)$ into \eqref{T2} produces
\begin{equation*}
\begin{aligned}
&\frac{\dot{\sigma}}{\sigma}(x)(\eta(x)-\eta(\lambda))-\dot{\eta}(x)+\dot{\eta}(\lambda)-m\left(\frac{(\eta(x)-\eta(\lambda))(\lambda-t) \rho(x)}{(x-\lambda)(x-t)}-\eta^{\prime}(\lambda) \rho(\lambda)\right)\\
&-m\left(\frac{(\lambda-t)}{(x-\lambda)^{2}}(2 \rho \eta(x)-2 \rho \eta(\lambda)-\left((\rho \eta)^{\prime}(x)+(\rho \eta)^{\prime}(\lambda)\right)(x-\lambda)\right)\\
&=\dfrac {n(n+2\alpha+1)}{4}\left(\frac{x-\lambda}{x-t}+\frac{mx(x-1)(\lambda-t)}{x-t}-m\lambda(\lambda-1)-m(\lambda-t)(x-\lambda)\right)\\
&=\dfrac {n(n+2\alpha+1)}{4}\left(\frac{x-\lambda}{x-t}-\frac{mt(t-1)(x-\lambda)}{x-t}\right)\\
&=\dfrac {n(n+2\alpha+1)}{4}\left(\frac{(1-mt(t-1))(x-\lambda)}{x-t}\right)=0,\\
\end{aligned}
\end{equation*}
where we used {\color{black}$m=\frac{1}{t(t-1)}$}.

Inserting $\sigma(x), \tau(x), \eta(x)$ into \eqref{T3} becomes
\begin{equation*}
\begin{aligned}
\frac{\dot{\sigma}}{\sigma}(x)-\frac{\dot{\sigma}}{\sigma}(\lambda)-m \rho(t)\frac{(x-\lambda)}{(x-t)(\lambda-t)}&=\frac{1}{\lambda-t}-\frac{1}{x-t}+\frac{mt(t-1)(\lambda-x)}{(\lambda-t)(x-t)} \\
&=\frac{(1-mt(t-1))(x-\lambda)}{(\lambda-t)(x-t)}=0.
\end{aligned}
\end{equation*}
So $\sigma(x), \tau(x), \eta(x)$ satisfy \eqref{T1}, \eqref{T2} and \eqref{T3} in the Lemma \ref{2}.

Define the compatibility functions
$$
a(t, \lambda, x)=-\frac{x(x-1)(\lambda-t)}{t(t-1)(x-\lambda)}, \quad b(t, \lambda, \mu, x)=-\frac{\lambda\mu(\lambda-1)(\lambda-t)}{t(t-1)(x-\lambda)},
$$
and the Hamiltonian
\begin{equation*}
\begin{aligned}
t(t-1)H(t, \lambda, \mu):=&\lambda(\lambda-1)(\lambda-t)\mu^{2}+\big(-\dfrac{3}{2}(\lambda-1)(\lambda-t)+\alpha\lambda(\lambda-t)+\lambda(\lambda-1)\big) \mu\\
&-\dfrac {n(n+2\alpha+1)\lambda+n\sqrt{t}}{4}.
\end{aligned}
\end{equation*}
Using the similar procedure in Theorem \ref{1}, we will obtain a Painlev\'{e} VI as follows
\begin{equation*}
\begin{aligned}
\frac{\mathrm{d^{2}} \lambda}{\mathrm{d} t^{2}} =&\frac{1}{2}\left(\frac{1}{\lambda}+\frac{1}{\lambda-1}+\frac{1}{\lambda-t}\right)\left(\frac{d\lambda}{dt}\right)^2-
\left(\frac{1}{t}+\frac{1}{t-1}+\frac{1}{\lambda-t}\right)\frac{d\lambda}{dt}\\
&+\frac{\lambda(\lambda-1)(\lambda-t)}{t(t-1)^{2}}(\widehat{\alpha}+\beta\frac{t}{\lambda^2}+\gamma\frac{t-1}{(\lambda-1)^2}+\delta\frac{t(t-1)}{(\lambda-t)^2},
\end{aligned}
\end{equation*}
where $\widehat{\alpha}=\frac{1}{2}n(n+2\alpha+1)+\frac{(2\alpha-1)^2}{8}, \beta=-\frac{9}{8}, \gamma=\frac{\alpha^2}{2}, \delta=\frac{1}{2}$.

\end{proof}

\section{Double scaling analysis {\color{black}for the Hankel determinant associated with $w_{\mathrm{SPG}}$}}
In this section, we study the asymptotic behavior of $D_{n}(t)$ under different double scaling. We combine our problem with the Hankel determinant generated by a singularly perturbed Laguerre unitary ensemble. Using the results obtained by {\color{black}Lyu, Griffin and Chen} in [\cite{LJC}], we show the asymptotics of the scaled Hankel determinant.

From the orthogonality \eqref{1.1}, we have
\begin{equation*}
\begin{aligned}
h_{2m}(t)\delta_{2m,2n}&=\int_{-\infty}^{\infty}P_{2m}(z)P_{2n}(z)|z|^{\alpha}\mathrm{e}^{-(z^2+\frac{t}{z^2})}\theta(z^2-a^2)dz\\
&=2\int_{0}^{\infty}P_{2m}(z)P_{2n}(z)z^{\alpha}\mathrm{e}^{-(z^2+\frac{t}{z^2})}\theta(z^2-a^2)dz\\
&=\int_{0}^{\infty}P_{2m}(\sqrt{x})P_{2n}(\sqrt{x})x^{\frac{\alpha-1}{2}}\mathrm{e}^{-(x+\frac{t}{x})}\theta(x-a^2)dx\\
&=\int_{a^2}^{\infty}P_{2m}(\sqrt{x})P_{2n}(\sqrt{x})x^{\frac{\alpha-1}{2}}\mathrm{e}^{-(x+\frac{t}{x})}dx,
\end{aligned}
\end{equation*}
which implies
\begin{align*}
\widetilde{P}_{m}(x,\frac{\alpha-1}{2})=P_{2m}(\sqrt{x}).
\end{align*}
Particularly,
\begin{equation}\label{4.10}
\begin{aligned}
h_{2m}(t)&=\int_{a^2}^{\infty}\left[\widetilde{P}_{m}\big(x,\frac{\alpha-1}{2}\big)\right]^2x^{\frac{\alpha-1}{2}}\mathrm{e}^{-(x+\frac{t}{x})}dx\\
&=\widetilde{h}_{m}(t,\frac{\alpha-1}{2}).
\end{aligned}
\end{equation}
Similarly,
\begin{equation*}
\begin{aligned}
h_{2m+1}(t)\delta_{2m+1,2n+1}&=\int_{-\infty}^{\infty}P_{2m+1}(z)P_{2n+1}(z)|z|^{\alpha}\mathrm{e}^{-(z^2+\frac{t}{z^2})}\theta(z^2-a^2)dz\\
&=2\int_{0}^{\infty}P_{2m+1}(z)P_{2n+1}(z)z^{\alpha}\mathrm{e}^{-(z^2+\frac{t}{z^2})}\theta(z^2-a^2)dz\\
&=\int_{0}^{\infty}\frac{P_{2m+1}(\sqrt{x})}{\sqrt{x}}\frac{P_{2n+1}(\sqrt{x})}{\sqrt{x}}x^{\frac{\alpha+1}{2}}\mathrm{e}^{-(x+\frac{t}{x})}\theta(x-a^2)dx\\
&=\int_{a^2}^{\infty}\frac{P_{2m+1}(\sqrt{x})}{\sqrt{x}}\frac{P_{2n+1}(\sqrt{x})}{\sqrt{x}}x^{\frac{\alpha+1}{2}}\mathrm{e}^{-(x+\frac{t}{x})}dx,
\end{aligned}
\end{equation*}
which implies
\begin{align*}
\widetilde{P}_{m}\big(x,\frac{\alpha+1}{2}\big)=\frac{P_{2m+1}(\sqrt{x})}{\sqrt{x}}.
\end{align*}
Then
\begin{equation}\label{4.12}
\begin{aligned}
h_{2m+1}(t)&=\int_{a^2}^{\infty}\left[\widetilde{P}_{m}\big(x,\frac{\alpha+1}{2}\big)\right]^2x^{\frac{\alpha+1}{2}}\mathrm{e}^{-(x+\frac{t}{x})}dx\\
&=\widetilde{h}_{m}\big(t,\frac{\alpha+1}{2}\big).
\end{aligned}
\end{equation}
 Define the Hankel determinations $\widetilde{D}_{n}(x,t,\lambda)$ generated by $x^{\lambda}\mathrm{e}^{-(x+\frac{t}{x})}, x \in[0,\infty),\alpha>-1, t>0,$
\begin{equation*}
\widetilde{D}_{n}(x,t,\lambda):=\mathrm{det}\bigg(\int_{a^2}^{\infty}x^{i+j}x^{\lambda}\mathrm{e}^{-(x+\frac{t}{x})}dx\bigg)_{i,j=0}^{n-1}.
\end{equation*}
The orthogonality goes
\begin{equation*}
\int^{\infty}_{0}\widetilde{P}_{i}(x)\widetilde{P}_{j}(x)x^{\lambda}\mathrm{e}^{-(x+\frac{t}{x})}dx=\widetilde{h}_i\delta_{ij},
\end{equation*}
where $\widetilde{P}_{i}(x)$ are the monic polynomials of degree $i$ orthogonal with respect to the weight $x^{\lambda}\mathrm{e}^{-(x+\frac{t}{x})}$.

\begin{thm}
Let
\begin{equation*}
\Delta_{1}(s,t):=\lim_{n\rightarrow\infty}\frac{D_{2n}(\frac{s}{4n},\frac{t}{2n+1+\alpha})}{D_{2n}(0,0)},\quad
\Delta_{2}(s,t):=\lim_{n\rightarrow\infty}\frac{D_{2n+1}(\frac{s}{4n},\frac{t}{2n+1+\alpha})}{D_{2n+1}(0,0)}.
\end{equation*}
For $s\rightarrow\infty$ and fixed $t$,
\begin{equation*}
\begin{aligned}
\ln\Delta_{1}(s)&=\ln\Delta_{2}(s)=\ln\left[ G(\frac{\alpha+3}{2})G\big(\frac{\alpha+1}{2}\big)\right]-\frac{\alpha}{2}\ln(2\pi)-\frac{s}{2}+\alpha\sqrt{s}\\
&-\frac{\alpha^2+1}{8}\mathrm{ln}s+\bigg(\frac{\alpha}{8}-2t\bigg)\frac{1}{\sqrt{s}}+\frac{\alpha^2+16\alpha t+1}{32s}+\bigg(\frac{\alpha^3}{96}+\frac{7\alpha}{128}
+\frac{t}{4}\bigg)s^{-\frac{3}{2}}\\
&+\bigg(\frac{\alpha^4}{256}+\frac{15\alpha^2}{256}+\frac{\alpha t}{8}+\frac{5}{128}\bigg)s^{-2}+\mathcal{O}\bigg(s^{-\frac{5}{2}}\bigg).
\end{aligned}
\end{equation*}

\end{thm}

\begin{proof}
Using \eqref{1.8}, we find the following relations from \eqref{4.10} and \eqref{4.12} :
\begin{subequations}
\begin{align}
D_{2n}(s,t)&=\widetilde{D}_{n}\big(s,t,\frac{\alpha+1}{2}\big)\widetilde{D}_{n}\big(t,\frac{\alpha-1}{2}\big),\label{4.15}\\
D_{2n+1}(s,t)&=\widetilde{D}_{n}\big(s,t,\frac{\alpha+1}{2}\big)\widetilde{D}_{n+1}\big(t,\frac{\alpha-1}{2}\big).\label{4.16}
\end{align}
\end{subequations}

From \eqref{4.15} and \eqref{4.16}, we obtain
\begin{equation*}
\Delta_{1}(s,t)=\Delta_{2}(s,t)=\Delta\big(s,t,\frac{\alpha-1}{2}\big)\Delta\big(s,t,\frac{\alpha+1}{2}\big),
\end{equation*}
where
\begin{equation*}
\Delta(s,t,\lambda):=\lim_{n\rightarrow\infty}\frac{\widetilde{D}_{n}(\frac{s}{4n},\frac{t}{2n+1+\lambda},\lambda)}{\widetilde{D}_{n}(0,0,\lambda)}.
\end{equation*}
Here, $\Delta(s,t,\lambda)$ has the following expansion in [\cite{LJC}], for $s\rightarrow\infty$ and fixed $t$,
\begin{equation}\label{4.44}
\begin{aligned}
\ln\Delta(s,t,\lambda)&=C\bigg(\frac{t}{s},\lambda\bigg)-\frac{s}{4}+\lambda\sqrt{s}-\frac{\lambda^2}{4}\mathrm{ln}s\\
&+\bigg(\frac{\lambda}{8}-t\bigg)\frac{1}{\sqrt{s}}+\frac{\lambda^2}{16s}+\bigg(\frac{\lambda^3}{24}+\frac{3\lambda}{128}+\frac{t}{8}\bigg)s^{-\frac{2}{3}}\\
&+\bigg(\frac{\lambda^4}{32}+\frac{9\lambda^2}{128}+\frac{\lambda t}{8}\bigg)s^{-2}+\mathcal{O}\bigg(s^{-\frac{5}{3}}\bigg),
\end{aligned}
\end{equation}
where $C(\frac{s}{t},\lambda)$ has the form of
\begin{equation*}
C\bigg(\frac{s}{t},\lambda\bigg)=\ln \frac{G(\lambda+1)}{(2\pi)^{\frac{\lambda}{2}}}+\frac{\lambda t}{2s}.
\end{equation*}
Then
\begin{equation}\label{4.20}
\begin{aligned}
C\left(\frac{s}{t},\frac{\alpha+1}{2}\right)+C\left(\frac{s}{t},\frac{\alpha-1}{2}\right)&=\ln \frac{G(\frac{\alpha+1}{2}+1)}{(2\pi)^{\frac{\alpha+1}{4}}}+\ln \frac{G(\frac{\alpha-1}{2}+1)}{(2\pi)^{\frac{\alpha-1}{4}}}+\frac{(\alpha+1)t}{4s}+\frac{(\alpha-1)t}{4s}\\
&=\ln \bigg[G\left(\frac{\alpha+3}{2}\right)G\left(\frac{\alpha+1}{2}\right)\bigg]-\frac{\alpha}{2}\ln(2\pi)+\frac{\alpha t}{2s},
\end{aligned}
\end{equation}
where $G(\cdot)$ representing the Barnes G-function [\cite{Voros}]. According to \eqref{4.44} and \eqref{4.20}, we obtain the result.
\end{proof}

\begin{thm}
For $s\rightarrow0^{+}$ and fixed $t>0$,
\begin{equation*}
\begin{aligned}
\ln\Delta_{1}(s,t)=\ln\Delta_{2}(s,t)=&=C\bigg(\frac{t}{s},\lambda\bigg)-\frac{2t}{\sqrt{s}}+\frac{\alpha^2}{8}\mathrm{ln}s+\alpha\sqrt{s}-\frac{s}{2}+\frac{s^{\frac{3}{2}}}{4t}\\
&-\frac{\alpha s^{\frac{5}{2}}}{8t^2}+\frac{s^3}{8t^2}+\bigg(\frac{\alpha^2}{8}-\frac{23}{128}\bigg)\frac{s^{\frac{7}{2}}}{t^{3}}-\frac{\alpha s^4}{8t^3}+\mathcal{O}\bigg(s^{\frac{9}{2}}\bigg).
\end{aligned}
\end{equation*}

\end{thm}

\begin{proof}
Here, $\Delta(s,t,\lambda)$ has the following expansion in [\cite{LJC}], for $s\rightarrow0^{+}$ and fixed $t>0$,
\begin{equation*}
\begin{aligned}
\ln\Delta(s,t,\lambda)&=C\bigg(\frac{t}{s},\lambda\bigg)-\frac{t}{\sqrt{s}}+\frac{1-4\lambda^2}{16}\mathrm{ln}s+\lambda\sqrt{s}-\frac{s}{4}+\frac{s^{\frac{3}{2}}}{8t}\\
&-\frac{\lambda s^{\frac{5}{2}}}{8t^2}+\frac{s^3}{16t^2}+\bigg(\frac{\lambda^2}{8}-\frac{27}{128}\bigg)\frac{s^{\frac{7}{2}}}{t^{3}}-\frac{\lambda s^4}{8t^3}+\mathcal{O}\bigg(s^{\frac{9}{2}}\bigg).
\end{aligned}
\end{equation*}
Because of \eqref{4.20} and the above expansion, we complete the proof.
\end{proof}

\begin{thm}
For $t\rightarrow\infty$ and fixed $s>0$,
\begin{equation*}
\begin{aligned}
\ln\Delta_{1}(s,t)&=\ln\Delta_{2}(s,t)=C\bigg(\frac{t}{s},\lambda\bigg)-\frac{2t}{\sqrt{s}}-\frac{s}{2}+\alpha\sqrt{s}-\frac{\alpha^2}{8}\mathrm{ln}s+\frac{s^{\frac{3}{2}}}{4t}\\
&+\frac{s^{\frac{5}{2}}}{8t^2}(s^{\frac{1}{2}}-\alpha)
+\frac{s^{\frac{7}{2}}}{192t^{3}}(16s-24\alpha\sqrt{s}+12\alpha^2-69)\\
&+\frac{s^{\frac{9}{2}}}{128t^4}(8s^{\frac{3}{2}}-16\alpha s+12\alpha^2\sqrt{s}-96\sqrt{s}-4\alpha^3+87\alpha)+\mathcal{O}\bigg(s^{\frac{9}{2}}\bigg).
\end{aligned}
\end{equation*}

\end{thm}

\begin{proof}
Here, $\Delta(s,t,\lambda)$ has the following expansion in [\cite{LJC}], for $t\rightarrow\infty$ and fixed $s>0$,
\begin{equation*}
\begin{aligned}
\ln\Delta(s,t,\lambda)=&C\bigg(\frac{t}{s},\lambda\bigg)-\frac{t}{\sqrt{s}}-\frac{s}{4}+\lambda\sqrt{s}+\frac{1-4\lambda^2}{16}\mathrm{ln}s+\frac{s^{\frac{3}{2}}}{8t}\\
&-\frac{\lambda s^{\frac{5}{2}}}{8t^2}+\frac{s^3}{16t^2}+\frac{s^{\frac{7}{2}}}{24t^{3}}\bigg(s-3\lambda\sqrt{s}+3\lambda^2-\frac{81}{16}\bigg)\\
&+\frac{s^{\frac{9}{2}}}{32t^4}\bigg(s^{\frac{3}{2}}-4\lambda s+\big(6\lambda^2-\frac{27 }{2}\big)\sqrt{s}-4\lambda^3+\frac{99\lambda}{4}\bigg)+\mathcal{O}\bigg(s^{\frac{9}{2}}\bigg).
\end{aligned}
\end{equation*}
Substituting \eqref{4.20} into the above expansion, we readily establish the theorem.
\end{proof}
\section{Conclusion}
We considered four kinds of orthogonal polynomials in this work. A unified form of differential equations satisfied by the polynomials are presented. It was shown that the differential equations will reduce to Heun equations as $n$ tends to infinity through suitable transformations and scalings. Based on these facts, we transformed a double-confluent Heun equation, two bi-confluent Heun equations, and a general Heun equation, associated with the polynomials, into Painlev\'{e} III$'$ and Painlev\'{e} VI, respectively. It is interesting that the Painlev\'{e} equations we got are same as the results obtained in [\cite{Clarkson}], [\cite{MC}] and [\cite{YLZC,MPAG2}], which are derived from the three term recurrence coefficients or some other auxiliaries but not orthogonal polynomials. Finally, we obtained the asymptotic behavior of $D_{n}(t)$ associated with $w_{\mathrm{SPG}}$ under different double scaling.

\section{Acknowledgements}

We firstly express our sincere thanks to Prof. Y. Chen for his enthusiastic help and valuable discussions.

M. Zhu acknowledges the support of the National Natural Science Foundation of China under Grant No. 12201333, the Natural Science Foundation of Shandong Province No. ZR2021QA034, the Foundation of Shandong Provincial Qingchuang Research Team, No. 2023KJ135.

C. Li is supported by the National Natural Science Foundation of China under Grant No. 12071237.

\section{Conflict of Interest}
The authors have no conflicts to disclose.

\end{document}